\documentclass[reqno]{article}
\usepackage{amsmath,amsfonts,amssymb,amsthm}

\newtheorem{theorem}{Theorem}
\newtheorem{lemma}[theorem]{Lemma}

\newtheorem{corollary}[theorem]{Corollary}

\theoremstyle{definition}
\newtheorem{remark}[theorem]{Remark}

\newcommand\noproof{\hfill$\Box$}

\renewcommand\le{\leqslant}
\renewcommand\ge{\geqslant}

\newcommand\eps{\varepsilon}
\newcommand\Ga{\Gamma}

\renewcommand\Pr{{\mathbb P}}
\newcommand\E{{\mathbb E}}

\newcommand\op{o_{\mathrm{p}}}

\newcommand\cF{\mathcal{F}}

\newcommand\bb[1]{\bigl(#1\bigr)}
\newcommand\Bb[1]{\Bigl(#1\Bigr)}

\newcommand\bigabs[1]{\bigl|#1\bigr|}
\newcommand\lrabs[1]{\left|#1\right|}
\newcommand\cE{\mathcal{E}}
\newcommand\cH{\mathcal{H}}
\newcommand\cG{\mathcal{G}}
\newcommand\cP{\mathcal{P}}
\newcommand\cM{\mathcal{M}}
\newcommand\cS{\mathcal{S}}
\newcommand\cD{\mathcal{D}}
\newcommand\cQ{\mathcal{Q}}
\newcommand\cT{{\mathcal{T}}}

\newcommand\pto{\overset{\mathrm{p}}{\to}}

\newcommand\ceil[1]{\lceil#1\rceil}

\newcommand\ds{{\bf d}}
\newcommand\dconf{d_{\mathrm{conf}}}

\newcommand\ini{i\hspace{0.1pt}n_i}
\newcommand\iri{i\hspace{0.05pt}r_i}
\newcommand\Gm{G^*}
\newcommand\pc{p_{\mathrm{c}}}
\newcommand\pairing{\pi}
\newcommand\vs{S}

\begin{document}
\title{An old approach to the giant component problem}
\author{B\'ela Bollob\'as%
\thanks{Department of Pure Mathematics and Mathematical Statistics,
Wilberforce Road, Cambridge CB3\thinspace0WB, UK; {\em and\/}
Department of Mathematical Sciences, University of Memphis, Memphis TN 38152,
USA; {\em and\/} London Institute for Mathematical Sciences, 35a South St.,
Mayfair, London W1K\thinspace2XF, UK.
E-mail: {\tt b.bollobas@dpmms.cam.ac.uk}.}
\thanks{Research supported in part by NSF grant DMS~1301614 and MULTIPLEX no.\ 317532.}
\and Oliver Riordan%
\thanks{Mathematical Institute, University of Oxford, 24--29 St Giles', Oxford OX1 3LB, UK.
E-mail: {\tt riordan@maths.ox.ac.uk}.}}
\date{October 4, 2012; revised February 15, 2015}
\maketitle

\begin{abstract}
In 1998, Molloy and Reed showed that, under suitable conditions, if a sequence $\ds_n$
of degree sequences converges to a probability distribution $D$, then the
proportion of vertices in the largest component of the
random graph associated to $\ds_n$
is asymptotically $\rho(D)$, where $\rho(D)$ is a constant defined by the
solution to certain equations that can be interpreted as the survival probability
of a branching process associated to $D$. There have been a number of papers
strengthening this result in various ways; here we prove a strong form
of the result (with exponential bounds on the probability of large deviations)
under minimal conditions. 
\end{abstract}

\section{Introduction and results}

By a {\em degree sequence} $\ds$ we mean a finite sequence $(d_1,\ldots,d_n)$ of
non-negative integers with even sum. The {\em length} $|\ds|$ of
$\ds=(d_i)_{i=1}^n$ is the number $n$ of terms, and the \emph{size}
$m(\ds)=\tfrac{1}{2}\sum_i d_i$ is half the sum of the terms.
We write $G_\ds$ for the \emph{random (simple)
graph with degree sequence $\ds$}, i.e., a graph with vertex set
$[n]=\{1,2,\ldots,n\}$ in which each vertex $i$ has degree $d_i$, chosen
uniformly at random from the set of all such graphs (assuming this set is
non-empty). As usual, in studying $G_\ds$ we also consider the corresponding
random \emph{configuration multigraph} $\Gm_\ds$, introduced
in~\cite{BB_config}, obtained by associating a set of $d_i$ stubs to each vertex
$i$, selecting a uniformly random pairing of the (disjoint) union of these sets,
and interpreting each paired pair of stubs as leading to an edge in the natural
way. Note that these graphs have $|\ds|$ vertices and $m(\ds)$ edges.

Let $\cD$ denote the set of probability distributions $D$ on the non-negative
integers with $0<\E(D)<\infty$. We usually write $D\in \cD$
as $D=(r_0,r_1,\ldots)$, where, abusing notation by also writing $D$ for a random
variable with distribution $D$,
$r_i=\Pr(D=i)$. One of the basic questions concerning
the random graph models just described is the following:
under what conditions does convergence of $\ds_n$ to $D$ imply that the
asymptotic behaviour of $G_{\ds_n}$ (or $\Gm_{\ds_n}$) is captured by $D$?
Here the behaviour we are interested in is the distribution of the component sizes,
and most particularly the number $L_1$ of vertices in the (a if there is a tie)
largest component.

Let
\[
 n_i(\ds) = \bigabs{\{j: d_j=i\}}
\]
denote the number of times a particular degree $i$ occurs in $\ds$,
so
\[
 m(\ds)=\frac{1}{2}\sum_{j=1}^{|\ds|} d_j = \frac{1}{2}\sum_{i=0}^\infty \ini(\ds).
\]
The basic assumptions made in all existing results of this type are that
\begin{equation}\label{dcn1}
  \lim_{n\to\infty} \frac{n_i(\ds_n)}{|\ds_n|} = r_i
\end{equation}
for each $i$, that
\begin{equation}\label{dcn2}
 \frac{m(\ds_n)}{|\ds_n|} \to \frac{\E(D)}{2} = \frac{1}{2}\sum_{i=0}^\infty \iri
\end{equation}
as $n\to\infty$, and of course that $|\ds_n|\to\infty$. (Often, one takes
$|\ds_n|=n$, which loses no generality.) We shall say that $\ds_n$
\emph{converges} to $D$, and write $\ds_n\to D$, if these conditions hold.

Condition \eqref{dcn1} ensures that $D$ captures the asymptotic proportion of
vertices of each fixed degree, and condition \eqref{dcn2} that the (rescaled)
number of edges is related to $D$ in the natural way. Note that if we write
$D_n$ for the distribution of a randomly chosen element of $\ds_n$ (i.e., the
degree of a random vertex of $G_{\ds_n}$), then \eqref{dcn1} asserts that $D_n$
converges to $D$ in distribution. Condition \eqref{dcn2} asserts that
$\E(D_n)\to \E(D)<\infty$, which (given \eqref{dcn1}) is equivalent to uniform
integrability of the $D_n$.

To see why \eqref{dcn2} is necessary, consider $\ds_n$ consisting of one vertex
of degree $n-1$ and $n-1$ of degree $1$, contrasted (for $n$ even) with $\ds_n'$ in which all
$n$ degrees are equal to $1$.  In both cases \eqref{dcn1} holds with $r_1=1$ and
all other $r_i=0$, but the component structures of $G_{\ds_n}$ and $G_{\ds_n'}$
are very different -- one is a star, and the other a matching.  (There is a
similar but less extreme difference between $\Gm_{\ds_n}$ and $\Gm_{\ds_n'}$.)

As usual, we write $L_i(G)$ for the number of vertices in the $i$th largest component
of a graph $G$. We also write $N_k(G)$ for the number of vertices in $k$-vertex components.
The next result involves constants $\rho_k(D)$ and $\rho(D)$ whose definitions we
postpone to Section~\ref{bp} (see \eqref{rhokdef}). In fact, $\rho(D)$ is the same
as the quantity $\eps_D$ appearing in~\cite{MolloyReed2}, although our definition of it is different.
We write $\pto$ to denote convergence in probability.

\begin{theorem}\label{th1}
Let $D$ be a probability distribution on the non-negative integers with $0<\E(D)<\infty$,
and let $\ds_n$ be a sequence of degree sequences converging to $D$ in the sense
that \eqref{dcn1} and \eqref{dcn2} hold and $|\ds_n|\to\infty$. Then 
\[
 N_k(G_{\ds_n})/|\ds_n| \pto \rho_k(D)
\]
for each fixed $k$. If $\Pr(D\ge 3)>0$, then we also have
\[
 L_1(G_{\ds_n})/|\ds_n| \pto \rho(D)
\]
and $L_2(G_{\ds_n})/|\ds_n|\pto 0$.

Furthermore, the same conclusions hold with $G_{\ds_n}$ replaced by $\Gm_{\ds_n}$.
\end{theorem}

The first result of this type was proved by Molloy and Reed~\cite{MolloyReed2}, building
on their work in~\cite{MolloyReed1}. This result required additional conditions:
taking $|\ds_n|=n$,
they assumed in particular that the maximum degree in $\ds_n$ is
$o(n^{1/4-\eps})$ for some $\eps>0$. Note that \eqref{dcn2} implies only
that the maximum degree is $o(n)$: adding a single vertex with degree
(approximately) $n/\log\log n$, say, does not affect convergence in our sense.

The results of~\cite{MolloyReed2} have been strengthened in a number of ways. One main
direction is to improve, or even study the distribution of, the error term in
the result $L_1=\rho(D)n +\op(n)$, sometimes imposing extra assumptions; see
Kang and Seierstad~\cite{KS}, Pittel~\cite{Pittel_rrg}, Janson and
Luczak~\cite{JL_new}, Riordan~\cite{norm2} and Hatami and Molloy~\cite{HatamiMolloy}, for example.
In the other direction,
one can ask for the same conclusion but with less restrictive assumptions;
here Janson and Luczak~\cite{JL_new}
prove a version of Theorem~\ref{th1}
with the condition that the sum of the squares of the degrees is at most
a constant times the number of vertices. They also prove the (easier) \emph{multigraph}
part of Theorem~\ref{th1} under exactly our conditions (see their Remark 2.6),
but using a very different method.

We shall in fact prove a much stronger form of Theorem~\ref{th1},
Theorem~\ref{th2} below; the reason for postponing the statement is that it
is a little more awkward: instead of convergence, we need to work with neighbourhoods.
Given $D\in\cD$ and a degree sequence $\ds$, writing $r_i=\Pr(D=i)$ as usual,
set
\begin{equation}\label{d10def}
 \dconf^0(\ds,D) = \sum_{i=1}^\infty\lrabs{i\frac{n_i(\ds)}{|\ds|} - \iri},
\end{equation}
so $\dconf^0$ is a form of the $\ell_1$ metric, and define the \emph{configuration distance}
between $\ds$ and $D$ to be
\begin{equation}\label{d11def}
 \dconf(\ds,D) = \max\{ \dconf^0(\ds,D), 1/|\ds| \}.
\end{equation}
The $1/|\ds|$ term in \eqref{d11def} ensures that $\dconf(\ds,D)\to 0$ if and only
if $\dconf^0(\ds,D)\to 0$ and $|\ds|\to\infty$, and avoids having to write `and $|\ds|\ge n_0$'
in many results below; this is a convenience rather than an essential part of the definition.

It is easy to check that, for $D\in \cD$,
\begin{equation}\label{coneq}
 \ds_n\to D  \Longleftrightarrow \dconf(\ds_n,D)\to 0.
\end{equation}
Indeed, if $\dconf(\ds_n,D)\to 0$, then certainly $|\ds_n|\to\infty$.
Also, $\dconf^0(\ds_n,D)\to 0$, which trivially implies \eqref{dcn1},
and implies \eqref{dcn2} by the triangle inequality.
Conversely, suppose that $\ds_n\to D$, and let $\eps>0$.
Since $\sum_i \iri=\E(D)<\infty$, there is some $C=C(\eps)$
such that $\sum_{i<C} \iri\ge \E(D)-\eps$, and so
\begin{equation}\label{iritail}
 \sum_{i\ge C} \iri \le \eps.
\end{equation}
For $n$ large enough, \eqref{dcn1} gives
\begin{equation}\label{niri}
 |\ini(\ds_n)/|\ds_n|-\iri|< \eps/C
\end{equation}
for all $i<C$. Hence
\[
 \sum_{i<C} \ini(\ds_n)/|\ds_n| \ge \sum_{i<C} \iri -\eps \ge \E(D)-2\eps.
\]
Using \eqref{dcn2} it follows that $\sum_{i\ge C} \ini(\ds_n)/|\ds_n|\le 3\eps$ if $n$ is large.
This, \eqref{iritail} and \eqref{niri} imply that $\dconf^0(\ds_n,D)\le 5\eps$. Since $\ds_n\to D$ implies $|\ds_n|\to\infty$
by definition, and $\eps$ was arbitrary, it follows that $\dconf(\ds_n,D)\to 0$.

Let us state for future reference a consequence of the argument just given: if
$\ds_n\to D$ then
\begin{equation}\label{dctail}
 \forall \eps>0\ \exists C\ \forall n\ \sum_{i\ge C} \ini(\ds_n)\le \eps|\ds_n|.
\end{equation}
Writing $\ds_n=(d_1^{(n)},\ldots,d_{\ell_n}^{(n)})$,
\eqref{dctail} can be written as 
\[
 \forall \eps>0\ \exists C\ \forall n\ \sum_{j\,:\,d_j^{(n)}\ge C} d_j^{(n)} \le \eps|\ds_n|.
\]
Informally, this condition says
that a random edge has only a small probability of being attached to a vertex of very high degree.
A rather trivial consequence of \eqref{dctail} is that, writing $\Delta(\ds)$ for
the maximum degree appearing in a degree sequence $\ds$, if $\ds_n\to D$ then
$\Delta(\ds_n) = o(|\ds_n|)$.
In terms of the metric, the equivalent of \eqref{dctail} is the observation that
\begin{equation}\label{dconftail}
 \forall D\in\cD, \eps>0 \ \exists C,\delta: \dconf(\ds,D)<\delta \Longrightarrow
 \sum_{i\ge C}\ini(\ds) \le \eps |\ds|.
\end{equation}
To see this, simply choose $C$ such that $\sum_{i\ge C} i\Pr(D=i)<\eps/2$,
and take $\delta=\eps/2$.

\begin{theorem}\label{th2}
Let $D\in\cD$, and let $\eps>0$. For each $k\ge 1$ there exists $\delta>0$ (depending
on $D, \eps$ and $k$) such that if $\dconf(\ds,D)<\delta$,
then
\begin{equation}\label{th2k}
 \Pr\bb{ |N_k(G_\ds) - \rho_k(D)n| \ge \eps n } \le e^{-\delta n},
\end{equation}
where $n=|\ds|$.
Moreover, if $\Pr(D\ge 3)>0$, then there exists $\delta>0$ (depending
on $D$ and $\eps$) such that if $\dconf(\ds,D)<\delta$
then
\[
 \Pr\Bb{ \bigabs{L_1(G_\ds)-\rho(D) n} \ge \eps n} \le e^{-\delta n}
\]
and
\[
 \Pr\bb{ L_2(G_\ds) \ge \eps n} \le e^{-\delta n}.
\]
Furthermore, the same conclusions hold if $G_\ds$ is replaced by $\Gm_\ds$.
\end{theorem}

Using \eqref{coneq}, it is easy to check that Theorem~\ref{th2} does indeed
strengthen Theorem~\ref{th1}. The main reason for proving the stronger bounds in
Theorem~\ref{th2} is that we need them for the configuration multigraph model
$\Gm_\ds$ in order to prove even Theorem~\ref{th1} for the simple random graph
$G_\ds$. Of course, they are also nice to have!

\begin{remark}
The condition $\Pr(D\ge 3)>0$ in Theorems~\ref{th1} and~\ref{th2}
is necessary for the conclusions; see Janson and Luczak~\cite[Remark 2.7]{JL_new} for a discussion
of the range of possible behaviours when $\Pr(D=2)=1$ (or $D$ is supported on $\{0,2\}$).
\end{remark}

The basic idea of the proof of Theorem~\ref{th1} is to use a (relatively) old
method. The first ingredient is to understand the local structure of $\Gm_\ds$;
this is very simple and can be expressed in a number of ways, most cleanly by
comparison with a branching process. This allows us to control the number of
vertices in small components. Then we use a version of the original sprinkling
argument of Erd\H os and R\'enyi~\cite{ER_giant} to show that almost all vertices in `large'
components are in a single giant component. Of course, sprinkling is more
complicated in the present model than in the original context. Also, to obtain
exponential error bounds we need a strong form of the branching process
approximation, which introduces some additional complications.  We shall show in
Section~\ref{sec_extras} that this approximation carries over to the giant
component: the number of vertices in the giant component with some `local'
property can be calculated in terms of the branching process.

Turning to the nitty-gritty, in the rest of the paper we use the following standard asymptotic notation:
given a sequence $E_n$ of events, we say that $E_n$ holds \emph{with high probability}
or \emph{whp} if $\Pr(E_n)\to 1$ as $n\to\infty$.
Given functions $f$ and $g$ of some parameter (usually $n$), $f=O(g)$ means $f$
is bounded by a constant times $g$, and $f=o(g)$ means that $f/g\to 0$ as the
parameter ($n$) tends to infinity.

Finally, before turning to the proofs, let us fix our formal notation for the 
configuration model: given a degree sequence $\ds$ of length $\ell$,
we take disjoint sets $F_1,\ldots,F_\ell$ with $|F_i|=d_i$, where
$F_i$ represents the `stubs' associated to vertex $i$. Then we take
a pairing (partition into 2-elements sets) $\pairing$ of $F=\bigcup_{i=1}^\ell F_i$
chosen uniformly at random, and set $\Gm_\ds=\phi_\ds(\pairing)$,
where $\phi_\ds$ maps a pairing $\pairing$ to a multi-graph on $[\ell]=\{1,2,\ldots,\ell\}$
by replacing each pair $\{a,b\}$ by an edge joining vertices $i$ and $j$
where $a\in F_i$ and $b\in F_j$, noting that $i=j$ is possible,
in which case the edge is a loop.

\section{Local approximation by a branching process}\label{bp}

Let $D=(r_0,r_1,\ldots)\in \cD$, so $D$ is a probability distribution on the non-negative integers with
$0<\E(D)<\infty$, and $r_i=\Pr(D=i)$.
For $i\ge 1$ let
\[
 q_i = \frac{\iri}{\sum_i \iri} = \frac{\iri}{\E(D)}.
\]
The distribution $D^*$ with $\Pr(D^*=i)=q_i$ is known as the \emph{size-biased} distribution
associated to $D$. In any graph $G$, if we pick a random edge
and then choose one of its endvertices $v$ at random, the distribution
of the degree of $v$ is the size-biased version of the degree distribution of $G$.
Let $Z_D=D^*-1$, so
\begin{equation}\label{Zdef}
 \Pr(Z_D=i) = \Pr(D^*=i+1) = \frac{(i+1)r_{i+1}}{\E(D)} = \frac{(i+1)\Pr(D=i+1)}{\E(D)}.
\end{equation}
Intuitively, $Z_D$ will correspond to the number of `new' edges we get to when
we follow a random edge to an endvertex.

Let $\cT^1=\cT^1_D$ be the Galton--Watson branching process with offspring
distribution $Z_D$, so $\cT^1$ is a random rooted tree in which the number of children of each vertex
has distribution $Z_D$, with these numbers independent.
Finally, let $\cT=\cT_D$ be the random rooted tree in which the degree of the root
has the distribution $D$, and, given the degree of the root,
the branches, i.e., the subtrees rooted at the children
of the root, form independent copies of $\cT^1$.

It is not hard to see that
if $\ds_n\to D$, then $\Gm_{\ds_n}$ `locally looks like' 
$\cT_D$; we shall make this precise in a moment. Let $|\cT_D|\le\infty$
denote the total number of vertices of $\cT_D$. Then the constants
$\rho_k$ and $\rho$ appearing in Theorems~\ref{th1} and~\ref{th2}
are
\begin{equation}\label{rhokdef}
 \rho_k(D) = \Pr( |\cT_D|=k ) \hbox{\quad and\quad} \rho(D)=\Pr(|\cT_D|=\infty).
\end{equation}

Given a graph $G$, for $v\in V(G)$ and $t\ge 0$,
let $\Ga_{\le t}(v)=\Ga_{\le t}^G(v)$ denote the subgraph
of $G$ induced by the vertices within (graph) distance $t$ of $v$,
regarded as a rooted graph with root $v$. Also, let $\cT_D|_t$ be the 
subtree of $\cT_D$ induced by the vertices within distance $t$ of the root.
The following lemma is a variant of an idea that is by now very much standard,
though perhaps not in exactly this form.

\begin{lemma}\label{l_couple}
Let $D\in \cD$ and suppose that $\ds_n\to D$.
Let $v$ be a vertex of $G=\Gm_{\ds_n}$ chosen uniformly at random.
Then we may couple the random graphs $\Ga_{\le t}^G(v)$ and $\cT_D|_t$
so that they are isomorphic as rooted graphs with probability $1-o(1)$
as $n\to\infty$.
\end{lemma}
\begin{proof}
As the argument is straightforward and standard we give only an outline.
The idea is to reveal $\Ga_{\le t}(v)$ step-by-step in the natural way,
coupling this process with revealing $\cT_D|_t$ step-by-step
so that for any fixed $j$, the probability of the coupling failing
at step $j$ is $o(1)$. Since, given any $\eps$, there
is some constant $J$ such that with probability at least $1-\eps$
the finite tree $\cT_D|_t$ has size at most $J$, this suffices
to prove the lemma.

To reveal $\Ga_{\le t}(v)$, first pick the random vertex $v$,
noting that by condition \eqref{dcn1} of the convergence $\ds_n\to D$,
the degree of $v$ can be coupled to agree with the degree 
of the root of $\cT_D$ with probability $1-o(1)$. Then go through the 
stubs associated to $v$ one-by-one, revealing their partners, and thus
the neighbours of $v$ (as well as any loops at $v$). Then reveal the partners
of the unpaired stubs associated to the neighbours of $v$, and so on.
The key fact is that the $j$th time we reveal the partner of an unpaired stub,
the probability that this is a `new' (not so far reached in the exploration)
vertex of degree $i$ is exactly
\[
 \frac{i (n_i(\ds_n)-u_{i,j})}{2m(\ds_n)+1-2j},
\]
where $u_{i,j}$ is the number of degree-$i$ vertices revealed so far.
For any fixed $j$, since $u_{i,j}\le j=O(1)$, this probability is $q_i+o(1)$.
Since $q_i$ is the probability that a vertex of $\cT_D$ other than the root
has degree $i$ (and hence $i-1$ children), it follows that the coupling
succeeds at step $j$ with probability $1-o(1)$, as required.
\end{proof}

\begin{corollary}\label{c_tree}
Let $D\in \cD$, suppose that $\ds_n\to D$, and let $t\ge 1$ be constant.
Let $v$ be a vertex of $G=\Gm_{\ds_n}$ chosen uniformly at random.
Then whp the neighbourhood $\Ga_{\le t}(v)$ of $v$ in $G$
is a tree.\noproof
\end{corollary}

Note that in many related situations, the equivalent of Corollary~\ref{c_tree}
is proved by considering the expected number of paths of length $k$ ending in a vertex
on a cycle of length $\ell$, showing that this expectation is $o(n)$
for $k$ and $\ell$ fixed. However, this requires some condition such as $\sum_i d_i^2=n^{1+o(1)}$,
which need not hold here -- it may be that $\Gm_\ds$ contains many (more than $n$) short
cycles, but these are all concentrated in the neighbourhoods of the few vertices
with largest degrees, so most vertices are far from them.

Let $\cP$ be a property of (locally finite) rooted graphs, i.e., a set of rooted graphs closed
under isomorphism. Often we think of $\cP$ as a property of
vertices $v$ of unrooted graphs $G$, by taking $v$ as the root; in either case
we write $(G,v)\in \cP$ to mean that the graph $G$ rooted at $v$ has property
$\cP$.  We write $N_\cP(G)$ for the number of vertices of $G$ with property
$\cP$.  Given $t\ge 1$, we say that $\cP$ is \emph{$t$-local} if whether $(G,v)$ has $\cP$
depends only on the rooted graph $\Ga_{\le t}^G(v)$. We call $\cP$ \emph{local} if it is
$t$-local for some $t$. Note that it makes sense to speak of our branching
process $\cT_D$ having property $\cP$, since $\cT_D$ is a rooted tree.  If $\cP$
is $t$-local, then whether $\cT_D$ has $\cP$ depends only on $\cT_D|_t$. 

Lemma~\ref{l_couple} immediately implies the following result, of which
Corollary~\ref{c_tree} is a special case.

\begin{corollary}\label{c_prop}
Let $\cP$ be a local property of rooted graphs,
let $D\in \cD$, suppose that $\ds_n\to D$,
and let $v$ be a vertex of $\Gm_{\ds_n}$ chosen uniformly at random.
Then
\[
 \Pr\bb{ (\Gm_{\ds_n},v)\in \cP } \to \Pr(\cT_D\in \cP)
\]
as $n\to\infty$.
Equivalently, $\E( N_\cP(\Gm_{\ds_n}) ) = \Pr(\cT_D\in \cP)|\ds_n|+o(|\ds_n|)$.\noproof
\end{corollary}

When we come to concentration, it will be convenient to work with a restatement
of this last corollary.
\begin{corollary}\label{c_prop1}
Let $\cP$ be a local property of rooted graphs, and
let $D\in \cD$ and $\eps>0$ be given. Then there exists $\delta>0$ such
that if $\dconf(\ds,D)<\delta$
then
\begin{equation}\label{propclose}
 \bigabs{\E( N_\cP(\Gm_\ds) ) - \Pr(\cT_D\in \cP)n} \le \eps n,
\end{equation}
where $n=|\ds|$.
\end{corollary}
\begin{proof}
Suppose not. Then for each $n$ there is a degree sequence $\ds_n$
with $\dconf(\ds_n,D)\le 1/n$ for which \eqref{propclose}
fails. Recalling \eqref{coneq},
applying Corollary~\ref{c_prop} to $(\ds_n)_{n=1}^\infty$ gives a contradiction.
\end{proof}

The key property to which we shall apply this result is the property
$\cP_k$ that the component of the root contains exactly $k$ vertices. Note that in this case
\begin{equation}\label{trans}
 N_{\cP_k}(G) = N_k(G) \hbox{\qquad and\qquad}\Pr(\cT_D\in\cP_k)=\rho_k(D).
\end{equation}

We can easily use the second moment method (exploring from two random
vertices $v$ and $w$) to prove that $N_\cP(\Gm_\ds)$ is concentrated in the sense
that $N_\cP(\Gm_{\ds_n})/n$ converges in probability when $\ds_n\to D$ with $|\ds_n|=n$.
Instead we use the Hoeffding--Azuma inequality to prove a stronger result.

Two configurations (pairings) $\pairing_1$ and $\pairing_2$ are \emph{related by a switching}
if $\pairing_2$ can be obtained from $\pairing_1$ by deleting two pairs $\{a,b\}$ and $\{c,d\}$
and inserting the pairs $\{a,c\}$ and $\{b,d\}$.
A function $f$ defined on pairings of some fixed set
is \emph{$C$-Lipschitz} if $|f(\pairing_1)-f(\pairing_2)|\le C$ whenever
$\pairing_1$ and $\pairing_2$ are related by a switching.
We shall use the following standard simple lemma.

\begin{lemma}\label{AHs}
Let $S$ be a set with size $2m$, and let $f$ be a $C$-Lipschitz
function of pairings of $S$.
If $\pairing$ is chosen uniformly at random from all pairings of $S$, then for any $t\ge 0$
we have
\[
 \Pr\Bb{ \bigabs{f(\pairing)-\E(f(\pairing))} \ge t} \le 2\exp(-t^2/(4C^2m)).
\]
\end{lemma}
\begin{proof}
Let $S=\{s_1,\ldots,s_{2m}\}$. Let us condition on the partners of
$s_1,\ldots,s_i$, writing $\Omega'$
for the set of all pairings consistent with the information revealed so far.
Now consider $s_{i+1}$. It may be that its partner
is determined, since it is paired to one of $s_1,\ldots,s_i$.
If not, for any possible partner $b$ let $\Omega'_b$ be the subset of $\Omega'$
containing all pairings in which $s_{i+1}$ is paired with $b$.
For distinct possible partners $b$ and $c$, there is a bijection
between $\Omega'_b$ and $\Omega'_c$ in which each $\pairing_1\in \Omega'_b$
is related to its image $\pairing_2$ by a switching: we simply switch
the pairs $\{s_{i+1},b\}$ and $\{c,d\}$ for $\{s_{i+1},c\}$ and $\{b,d\}$,
where $d$ is the partner of $c$ in $\pairing_1$ (and hence of $b$ in $\pairing_2$).

Write $\cF_i$ for the (finite) sigma-field generated by the random variables
listing the partners of $s_1,\ldots,s_i$.  The bijection
just given and the Lipschitz property of $f$ easily imply that
$\E( f(\pairing) \mid \cF_{i+1})$ is always within $C$ of $\E( f(\pairing)\mid \cF_i)$.
Thus the sequence $(X_i)_{i=0}^{2m}$ with $X_i=\E( f(\pairing)\mid \cF_i)$
is a martingale with differences bounded by $C$.
The result now follows from the Hoeffding--Azuma inequality,
noting that $X_0=\E(f(\pairing))$ and $X_{2m}= f(\pairing)$.
\end{proof}

Since $N_k(G)$ changes by at most $2k$ when an edge is added to or deleted
from a multigraph $G$, and a switching corresponds to deleting two edges and adding
two edges, $N_k(\Gm_\ds)$ is $8k$-Lipschitz as a function
of the pairing used to generate $\Gm_\ds$. (In fact, it is
$4k$-Lipschitz.) Thus Lemma~\ref{AHs} implies concentration
of $N_k(\Gm_\ds)=N_{\cP_k}(\Gm_\ds)$. Later we shall consider more general
properties than $\cP_k$, and then we must work harder to obtain
concentration results -- in general for a local
property $\cP$, there is no constant $C=C(\cP)$ such that $N_{\cP}(G)$
is $C$-Lipschitz. So we need to modify our properties slightly, 
to `avoid high-degree vertices'.

For $\Delta\ge 2$ and $t\ge 0$, let $\cM_{\Delta,t}$
be the property that every vertex within graph distance $t$ of the root has
degree at most $\Delta$. Note that $\cM_{\Delta,t}$ is $(t+1)$-local.
\begin{lemma}\label{Mlip}
Let $\cP$ be a $t$-local property, and let $\cQ=\cP\cap\cM_{\Delta,t}$.
Then the number $N_{\cQ}(G)$ of vertices of a multigraph $G$ with
property $\cQ$ changes by at most $4\Delta^t$ if a single
edge is added to or deleted from $G$. Furthermore,
$N_\cQ(G)$ is $16\Delta^t$-Lipschitz.
\end{lemma}
\begin{proof}
The effect of a switching on the corresponding configuration multigraph is to
delete two edges and then add two edges (perhaps between the same vertices). 
Thus it suffices to prove the first statement.

Let $v$ be a vertex of $G$ such that one of $(G,v)$ and $(G+e,v)$ has property $\cQ$
but the other does not. Note that since $\cM_{\Delta,t}$ is monotone decreasing,
$(G,v)\in \cM_{\Delta,t}$. If $e=xy$, then the graph distance
from $v$ to $\{x,y\}$ is the same in $G$ and in $G+e$. Clearly, this distance
is at most $t$; otherwise the presence of $e$ would not affect the property $\cQ$. 
Hence, in $G$, at least one endpoint of $e$ is within distance $t$ of $v$,
so $v$ is joined to an endpoint of $e$ by a path in $G$ of length at most $t$ in which
(since $(G,v)\in \cM_{\Delta,t}$)
each vertex has degree at most $\Delta$. Each endpoint of $e$ is the end of at most
$(1+\Delta+\cdots+\Delta^t)\le 2\Delta^t$ such paths, so there can be at
most $4\Delta^t$ vertices $v$ with the claimed property.
\end{proof}

The next lemma shows that provided we choose $\Delta$ large enough, there is no harm
in considering only vertices whose local neighbourhoods contain only vertices
with degree at most $\Delta$.

\begin{lemma}\label{nobig}
Let $D\in\cD$, $t\ge 0$ and $\eps>0$ be given. Then there exist $\delta>0$ and
an integer $\Delta$ such that
\[
 \Pr( \cT_D \hbox{ has }\cM_{\Delta,t} ) \ge 1-\eps/10
\]
and 
\begin{equation}\label{Mconc}
 \Pr\bb{ N_{\cM_{\Delta,t}}(\Gm_\ds) \le n-\eps n/2 } \le e^{-\delta n}
\end{equation}
whenever $\dconf(\ds,D)<\delta$, where $n=|\ds|$.
\end{lemma}
Thus for any given $t$ and $\eps$ there is a $\Delta$ such that with very high probability,
for $\dconf(\ds,D)$ small enough, at most $\eps n/2$ vertices of $\Gm_\ds$ are within
distance $t$ of a vertex with degree larger than $\Delta$.
\begin{proof}
The first statement is immediate from the fact that the random variable $M$ giving 
the maximum degree 
of any vertex of $\cT_D$ within distance $t$ of the root is always finite, so there is some
$\Delta$ such that $\Pr(M>\Delta)<\eps/10$. Corollary~\ref{c_prop1} implies that,
if $\delta$ is small enough, then $\dconf(\ds,D)<\delta$ implies
that $N=N_{\cM_{\Delta,t}}(\Gm_\ds)$ has expectation within
$\eps n/10$ of $n\Pr(\cT_D\in \cM_{\Delta,t})$, so $\E(N)\ge n-\eps n/5$.
By Lemma~\ref{Mlip}, applied with $\cP$ the `trivial' $t$-local
property that always holds, as a function of the pairing used to generate $\Gm_\ds$,
the quantity $N$ is $C$-Lipschitz
for some $C$. Now \eqref{Mconc} follows by Lemma~\ref{AHs}.
\end{proof}

We are now in a position to establish concentration of the number of vertices 
whose neighbourhoods have some local property.
\begin{theorem}\label{th_local1}
Let $\cP$ be a local property of rooted graphs, let $D\in \cD$ and let $\eps>0$.
There is some $\delta>0$ such that if $\dconf(\ds,D)<\delta$ then
\begin{equation}\label{totconc}
 \Pr\Bb{ \bigabs{
      N_\cP(\Gm_\ds)  - n\Pr(\cT_D \in \cP )
     } \ge \eps n }\le e^{-\delta n},
\end{equation}
where $n=|\ds|$.
\end{theorem}
\begin{proof}
Let $D\in \cD$, $\eps>0$ and a $t$-local property $\cP$ be given,
and let $\Delta$ be as in Lemma~\ref{nobig}.
Let us say that an event holds \emph{with very high probability} or \emph{wvhp}
if for some constant $\delta>0$
it has probability at least $1-e^{-\delta n}$ whenever $\dconf(\ds,D)<\delta$.
So in particular, Lemma~\ref{nobig} tells us that wvhp all but at most $\eps n/2$
vertices of $G=\Gm_\ds$ have property $\cM=\cM_{\Delta,t}$.

Let $N=N_{\cP}(G)$ be the number of vertices with property $\cP$,
let $B=n-N_{\cM}(G)$ be the number of `bad' vertices, i.e, ones not having property $\cM$,
and let $N'=N_{\cP\cap\cM}$ be the number of `good' vertices with property $\cP$. Then,
wvhp,
\[
 |N-N'|\le B \le \eps n/2.
\]
By the first part of Lemma~\ref{nobig}, we have
\[
 |\Pr(\cT_D\in \cP) - \Pr(\cT_D\in \cP\cap \cM)| \le \Pr(\cT_D\notin \cM) \le \eps /10.
\]
By Lemma~\ref{Mlip}, $N'$ is $C$-switching Lipschitz for some constant $C$, so by
Corollary~\ref{c_prop1} and Lemma~\ref{AHs}, we have that wvhp
\[
 |N'-n\Pr(\cT_D\in\cP\cap\cM)| \le \eps n/10,
\]
say. The last three displayed equations and the triangle inequality establish \eqref{totconc}. 
\end{proof}

\begin{corollary}\label{c_mrk_conc}
Let $D\in \cD$, and let $k\ge 1$ and $\eps>0$ be given. Then there exists $\delta>0$
such that if $\dconf(\ds,D)<\delta$ then
\begin{equation}\label{mrkc}
 \Pr\Bb{\bigabs{ N_k - \rho_k n} \ge \eps n} \le e^{-\delta n}
\end{equation}
where $n=|\ds|$, $N_k=N_k(\Gm_\ds)$ and $\rho_k=\Pr(|\cT_D|=k)$.
\end{corollary}
\begin{proof}
Recall \eqref{trans} and apply Theorem~\ref{th_local1} to the property $\cP_k$.
\end{proof}

This corollary proves the first statement \eqref{th2k} of Theorem~\ref{th2}, and
hence the corresponding statement in Theorem~\ref{th1}. One can obtain an explicit
constant in the exponential error probability in \eqref{mrkc} by using that
$N_k$ is $4k$-Lipschitz, but there does not seem to be much point.

To conclude this section, we note that, as usual, summing over $k'<k$ and subtracting from $n$,
bounds on $N_k$ with $k$ fixed give
bounds on $N_{\ge k}$ as well, where $N_{\ge k}(G)$ denotes the number
of vertices of a graph $G$ in components of order at least $k$.

\begin{lemma}\label{lNgekD}
Let $D\in\cD$, $\eps>0$ and $K$ be given.
There exist $k\ge K$ and $\delta>0$ such that if $\dconf(\ds,D)<\delta$ then
\begin{equation}\label{mroc}
 \Pr\Bb{ \bigabs{N_{\ge k} - \rho(D) n} \ge\eps n} \le e^{-\delta n},
\end{equation}
where $n=|\ds|$, $N_{\ge k}=N_{\ge k}(\Gm_\ds)$ and $\rho(D)=\Pr(\cT_D\text{ is infinite})$.
\end{lemma}
\begin{proof}
Since $\sum_k \rho_k(D)=\Pr(|\cT_D|<\infty)=1-\rho(D)$,
there is some $k\ge K$ such that $\sum_{k'=1}^{k-1} \rho_{k'}$ is within $\eps/2$
of $1-\rho(D)$. The result follows by applying Lemma~\ref{c_mrk_conc} for each $k'\le k-1$,
with $\eps/(2k)$ in place of $\eps$.
\end{proof}

As usual, the result for $k$ fixed extends to the case when $k\to\infty$
slowly, showing, roughly speaking, that the probability
that the branching process $\cT_D$ is infinite gives the asymptotic
proportion of vertices in `large' components.

\section{The survival probability $\rho(D)$}\label{sec_surv}

In this brief section we discuss the behaviour of the survival probability
$\rho(D)$ of the branching process $\cT_D$. The result below is needed
in the next section, but also helps to interpret Theorems~\ref{th1} and~\ref{th2}.

Recall that from generation $1$ onwards, $\cT_D$ behaves like the Galton--Watson
branching process $\cT_D^1$ with offspring distribution $Z_D$ defined by \eqref{Zdef},
and that $\cT_D$ simply consists of a random number $N$ of copies of $\cT_D^1$,
with $N$ having the distribution $D$.
\begin{theorem}\label{rhoD}
Let $D$ be any distribution on the non-negative integers with $\Pr(D\ge 3)>0$ and $\E(D)<\infty$.
Then $\rho(D)>0$ if and only if $\E(D(D-2))>0$.
Furthermore, writing $x_+$ for the largest solution in $[0,1]$ to
\begin{equation}\label{xdef}
 x = 1 - \sum_{i=1}^\infty \frac{\iri}{\E(D)} (1-x)^{i-1},
\end{equation}
where $r_i=\Pr(D=i)$, we have
\begin{equation}\label{rd}
 \rho(D) = 1- \sum_{i=0}^\infty r_i (1-x_+)^i.
\end{equation}

Finally, suppose that $D_1,D_2,\ldots$ are distributions on
the non-negative integers such that $D_n\to D$ in distribution
and $\E(D_n)\to \E(D)$. Then $\rho(D_n)\to\rho(D)$ as $n\to\infty$.
\end{theorem}
\begin{proof}
Standard results on Galton--Watson processes
tell us that the survival probability of $\cT_D^1$ is equal to $x_+$, the largest solution
in $[0,1]$ to \eqref{xdef}. Furthermore, since
$\Pr(D\ge 3)>0$ rules out the trivial case $\Pr(Z_D=1)=1$,
we have $x_+>0$ if and only if $\E(Z_D)>1$. Conditioning on the number $N$
of children of the root of $\cT_D$ gives \eqref{rd}
as an immediate consequence, and shows that
$\rho(D)>0$ if and only if $x_+>0$, i.e., if and only if $\E(Z_D)>1$.
Since $\E(Z_D)=\sum_i (i-1)\Pr(Z_D=i-1)=\sum_i i(i-1)r_i/\sum_i ir_i$,
this condition is equivalent to $\sum_i i(i-2)r_i>0$.

For the last part, define $Z_{D_n}$ from $D_n$ as in \eqref{Zdef},
i.e., by size-biasing and then subtracting 1.
Since $\Pr(D_n=i)\to r_i$ and $\E(D_n)\to \E(D)$,
we have $\Pr(Z_{D_n}= i)\to \Pr(Z_D=i)$.
Standard branching process results then imply
that the survival probability of $\cT^1_{D_n}$ converges
to that of $\cT^1_D$. Using \eqref{rd}, it follows easily that $\rho(D_n)\to \rho(D)$.
\end{proof}

\begin{remark}
The formulae above coincide (as they must) with those given by Molloy and
Reed~\cite{MolloyReed2} -- one can check that $x_+$ is equal to $1-\sqrt{1-2\alpha_D/K}$ in
their notation. They did not use the branching process
interpretation, however. In the notation of Janson and Luczak~\cite{JL_new},
$x_+$ is $1-\xi$, and $\rho(D)$ is $1-g(\xi)$.
\end{remark}

\section{Colouring and sprinkling}

Our next task is to use `sprinkling' to show that whp almost all vertices in
`large' components are in a single `giant' component.  In the original context
of the random graphs $G(n,p)$ and $G(n,m)$, sprinkling is very simple to
implement -- first include each edge independently with probability $p_1$, then
`sprinkle' in extra edges by including each edge not already present
independently with probability $p_2$, where $p_1+p_2-p_1p_1=p$. In the context
of the configuration model, there is no very simple analogue of this. Instead,
we will `thin' the random graph $\Gm_\ds$, and then put
back the deleted edges.

Given $0<p<1$, let $G'=\Gm_\ds[p]$ denote
the random subgraph of $G=\Gm_\ds$ obtained by retaining
each edge independently with probability $p$, 
and let $G''$ be the multigraph formed by the deleted edges,
so $V(G'')=V(G')=V(G)$ and $E(G)$ is the disjoint union of $E(G')$ and $E(G'')$.
Let $\ds'$ be the (random, of course)
degree sequence of $G'$, and $\ds''$ that of $G''$,
so $d'_i+d''_i=d_i$ for each vertex $i\in V(G)$. The following simple observation
is a key ingredient of the sprinkling argument.

\begin{lemma}\label{redblue}
For any $\ds$ and any $0<p<1$, the random graphs $G'$ and $G''$ are conditionally
independent given $\ds'$, having the distributions of $\Gm_{\ds'}$ and $\Gm_{\ds''}$
respectively.
\end{lemma}
\begin{proof}
This is essentially immediate from the definition of the configuration model: recall
that $G$ is defined from a pairing $\pairing$ of a set of $2m(\ds)$ stubs. Given this
pairing, colour each pair red with probability $p$ and blue otherwise, independently
of the others. Then we may take $G'$ to be given by the red pairs and $G''$ by the blue
pairs. Clearly, given the set of stubs in red pairs (which determines $\ds'$ and thus $\ds''$),
the pairing of these red stubs is uniformly random, and similarly for the blue stubs.
\end{proof}

Our next aim is to extend the coupling result Theorem~\ref{th_local1} to the
pair $(G',G'')$. First we need some definitions. We shall work with $2$-coloured multigraphs
(rather than coloured pairings as above).
Given a degree sequence $\ds$
and $0<p<1$, let $\Gm_\ds\{p\}$ denote the random \emph{coloured} graph
obtained by constructing $\Gm_\ds$ and then colouring the edges independently,
each red with probability $p$ and blue otherwise. Thus $G'=\Gm_\ds[p]$
may be viewed as the red subgraph of $\Gm_\ds\{p\}$.  Similarly, let
$\cT_D\{p\}$ be the random coloured rooted tree obtained from $\cT_D$ by
colouring each edge red with probability $p$ and blue otherwise, independently
of the others.

Given a probability distribution $D$ on the non-negative integers, and $0<p<1$,
let $D_p$ be the \emph{$p$-thinned} version of $D$, which may be defined by
taking a random set $X$ of size $D$ and selecting elements of $X$ independently
with probability $p$.
Then $D_p$ is the (overall) distribution of the number of
selected elements. To spell this out, and for later reference,
writing $r_i=\Pr(D=i)$ as usual, for $0\le i\le j$ let
\begin{equation}\label{rijdef}
 r_{ij} = r_j\binom{j}{i}p^i(1-p)^{j-i},
\end{equation}
and let
\begin{equation}\label{pi'}
 r_i' = \sum_{j\ge i} r_{ij}.
\end{equation}
Then
\begin{equation}\label{Dpdef}
  \Pr(D_p=i) = r_i'.
\end{equation}
It is a simple exercise in basic probability to check
that the operations of (i) $p$-thinning and (ii) size-biasing and then subtracting $1$
commute. A simple consequence of this is that the component of the red subgraph
of $\cT_D\{p\}$ containing the root has the same distribution as $\cT_{D_p}$.

The next result concerns `local properties of coloured rooted graphs', which are
defined in the obvious way.

\begin{theorem}\label{th_localc}
Let $\cP$ be a local property of coloured rooted graphs, let $D\in \cD$, let $\eps>0$ and
let $0<p<1$.
There is some $\delta>0$ such that if $\dconf(\ds,D)<\delta$ then
\begin{equation}\label{totconcc}
 \Pr\Bb{ \bigabs{
      N_\cP(\Gm_\ds\{p\})  - n\Pr(\cT_D\{p\} \in \cP )
     } \ge \eps n }\le e^{-\delta n},
\end{equation}
where $n=|\ds|$. Furthermore, if $\cQ$ is a local property of rooted graphs,
then there is some $\delta>0$ such that if $\dconf(\ds,D)<\delta$ then
\begin{equation}\label{conQ}
 \Pr\Bb{ \bigabs{
      N_\cQ(\Gm_\ds[p])  - n\Pr(\cT_{D_p} \in \cQ )
     } \ge \eps n }\le e^{-\delta n}.
\end{equation}
\end{theorem}
\begin{proof}
From the remarks above, it suffices to prove the first statement, \eqref{totconcc}.
Then \eqref{conQ} may be deduced by applying \eqref{totconcc} to the
local property $\cP$ that the component of the red graph containing the root has
property $\cQ$.
We give only an outline proof of \eqref{totconcc}, since the argument is a simple modification
of that of Theorem~\ref{th_local1}.

Firstly, the coloured analogue of Lemma~\ref{l_couple} follows from Lemma~\ref{l_couple}: when the 
coupling as uncoloured graphs succeeds, we may apply the same (random) colouring
to $\Ga_{\le t}(v)$ as to $\cT_D|_t$. Arguing as before, we deduce the coloured analogue of
Corollary~\ref{c_prop1}.
Now $N_{\cP}(\Gm_\ds\{p\})$
depends not only on the configuration, but also on the colouring. However,
passing to a property $\cQ=\cP\cap\cM_{\Delta,t}$ as in the proof of Theorem~\ref{th_localc},
by a variant of Lemma~\ref{Mlip}
we see that $N_{\cQ}$ changes by at most a constant (a) under a switching and (b) under
changing the colour of a single edge. Now we can apply the Hoeffding--Azuma inequality to a martingale
with $2m$ steps for the switchings and $m$ for the colour choices, where $m=m(\ds)$
is the number of edges of $\Gm_\ds$, to deduce concentration of $N_{\cP}(\Gm_\ds\{p\})$ and complete
the proof.
\end{proof}

Recall that $n_i=n_i(\ds)$ is the number of vertices with degree $i$ in $G=\Gm_\ds$.
Let $n_i'$ be the number of vertices with degree $i$ in the random subgraph $G'=G[p]$
defined earlier. Also, 
for $0\le i\le j$, let $n_{ij}$ be the number of vertices with degree $i$
in $G'$ and degree $j$ in $G$. Thus $n_i'=\sum_{j\ge i}n_{ij}$. Recall the definitions
\eqref{rijdef} and \eqref{pi'}; at an intuitive level these formulae give
the expected proportions of vertices of $G'$
having degree $i$ (for $r_i'$) and having degree $i$ in $G'$ and degree
$j$ in $G$, ignoring the effect of loops. Hence the next lemma comes as no surprise.

\begin{lemma}\label{pij}
Let $D\in \cD$ and $0<p<1$ be fixed.
Given $0\le i\le j$ and $\eps>0$
there exists $\delta>0$ such that if $\dconf(\ds,D)<\delta$ then
\[
 \Pr\bb{ |n_{ij}-r_{ij}n| \ge \eps n} \le e^{-\delta n}
\]
and
\[
 \Pr\bb{ |n_i'-r_i'n| \ge \eps n} \le e^{-\delta n},
\]
where $n=|\ds|$.
\end{lemma}
\begin{proof}
Apply Theorem~\ref{th_localc} to the $1$-local coloured rooted graph properties `the root
is incident with $j$ edges in total of which $i$ are red' for the first
statement, and `the root is incident with $i'$ red edges' for the second.
\end{proof}

Recall that $D_p$, the $p$-thinned version of the probability distribution $D$,
may be defined by \eqref{Dpdef}.
\begin{corollary}\label{cpij}
Given $D\in \cD$, $0<p<1$ and $\eps>0$ there exists $\delta>0$ such that,
if $\dconf(\ds,D)<\delta$, then
\[
 \Pr\Bb{\dconf(\ds',D_p)\ge \eps} \le e^{-\delta n},
\]
where $\ds'$ is the degree sequence of the random subgraph $G[p]$ of $G=\Gm_\ds$ and $n=|\ds|=|\ds'|$
is the number of vertices.
\end{corollary}
\begin{proof}
Since $\E(D)<\infty$ there is a constant $C$ such that $\sum_{i\ge C}i r_i<\eps/8$.
If $\delta$ is small enough, then $\dconf(\ds,D)<\delta$ implies $\sum_{i\ge C} \ini(\ds)< \eps n/4$.
Since $D$ stochastically dominates $D_p$, and
the degree of a vertex in our random subgraph $G'$ is at most its degree
in $G$, the corresponding bounds for $D_p$ and $n_i'=n_i(\ds_n')$ follow.
From the definition \eqref{d10def}, \eqref{d11def} of $\dconf$ it thus suffices 
to prove that for each fixed $i<C$ we have
$|n_i'-r_i' n|\le \eps/(2C^2)$ with sufficiently high probability;
this follows from Lemma~\ref{pij}.
\end{proof}

The next trivial lemma will be applied to the sprinkled edges.
\begin{lemma}\label{join}
Let $A$ and $B$ be disjoint sets of stubs in the configuration model associated to $\Gm_\ds$.
Then the probability that no stubs in $A$ are paired to stubs in $B$ is at most
$\exp(-|A||B|/(8m))$, where $m=m(\ds)$.
\end{lemma}
\begin{proof}
Assume without loss of generality that $|A|\le |B|$. Perform a sequence of
$\ceil{|A|/2}$ experiments, each consisting of choosing an as-yet-unpaired stub
in $A$ and revealing its partner. 
In the $i$th experiment, there are
at least $|B|-(\ceil{|A|/2}-1)\ge |B|-|A|/2\ge |B|/2$
unpaired stubs in $B$, so the probability of finding the partner in $B$ is at least
$(|B|/2)/(2m+1-2i) \ge |B|/(4m)$. Hence the probability that no partner in $B$
is found is at most $(1-|B|/(4m))^{|A|/2}\le \exp(-|A||B|/(8m))$.
\end{proof}

We are finally ready to prove the multigraph case of Theorem~\ref{th2},
where $G_\ds$ is replaced by $\Gm_\ds$.

\begin{proof}[Proof of Theorem~\ref{th2} for $\Gm_\ds$]
Let $L_i=L_i(\Gm_\ds)$ be the number of vertices in the $i$th largest component of $\Gm_\ds$.

Fix $D\in\cD$ and $\eps>0$. By Lemma~\ref{lNgekD} there are constants $k$ and $\delta>0$
such that if $\dconf(\ds,D)<\delta$, then
\[
 \Pr\bb{N_{\ge k}(G)\ge (\rho(D)+\eps/8)n}\le e^{-\delta n}.
\]
Since $L_1+L_2\le N_{\ge k}+2k$, if $n$ is large enough (which we can ensure by 
taking $\delta$ small enough) it follows that
\begin{equation}\label{L1L2ub}
 \Pr\Bb{L_1+L_2 \le (\rho(D)+\eps/4) n} \ge 1-e^{-\delta n}.
\end{equation}
To complete the proof, it suffices to show that
if $\dconf(\ds,D)<\delta$ then
\begin{equation}\label{MRaim}
 \Pr\Bb{L_1 \ge (\rho(D)-3\eps/4) n} \ge 1-e^{-\delta n}.
\end{equation}
Of course, this may require reducing $\delta$. Indeed,
\eqref{L1L2ub} and \eqref{MRaim} together
give high probability upper and lower bounds on $L_1$, and a high probability
upper bound on $L_2$. Since we have already proved \eqref{th2k} in Corollary~\ref{c_mrk_conc},
Theorem~\ref{th2} then follows.

As $p\to 1$, the probability distribution $D_p$ defined above converges
to $D$, both in distribution and (since $\E(D_p)\le \E(D)<\infty$) in expectation.
Hence, Theorem~\ref{rhoD} tells us that $\rho(D_p)\to \rho(D)$ as $p\to 1$.
(This is the only place in the argument where $\Pr(D\ge 3)>0$ is used.)
In particular, there is some $p<1$ such that
\[
 \rho(D_p) \ge \rho(D)- \eps/8.
\]
Let us fix such a $p$ for the rest of the proof. Also, fix an integer $t\ge 1$ such that
\[
 p^t \le \eps/20,
\]
set
\[
 K = 1+\Delta+\cdots +\Delta^{t-1}+1,
\]
and let
\begin{equation}\label{ag}
 \alpha = \frac{\eps}{40\Delta^t} \hbox{\qquad and\qquad}\gamma=\frac{\alpha^2}{8\E(D)}.
\end{equation}

We shall study the coloured random graph $\Gm_\ds\{p\}$ defined earlier,
obtained from $\Gm_\ds$ by colouring each edge red with probability $p$
and blue otherwise, independently of the others. As before, we write $G'=\Gm_\ds[p]$
for the red subgraph and $G''$ for the blue subgraph,
and $\ds'$ and $\ds''$ for the degree sequences of $G'$ and $G''$.
Recall that, by Lemma~\ref{redblue}, given $\ds'$, we can view $G'$ and $G''$
as independent configuration multigraphs.

Applying Lemma~\ref{lNgekD} to $G'$, we find that
there exist $k\ge \max\{K,2/\gamma\}$ and $\delta_1>0$ such that, writing
$\vs$ for the set of vertices in components of $G'$ with at least $k$ vertices,
we have
\[
 \Pr\Bb{ |\vs| \ge (\rho(D)-\eps/4)n \bigm| \ds' }
  \ge  \Pr\Bb{ |\vs| \ge (\rho(D_p)-\eps/8)n \bigm| \ds' } \ge 1- e^{-\delta_1 n}
\]
whenever $\dconf(\ds',D_p)<\delta_1$.
By Corollary~\ref{cpij} there is a $\delta_2>0$ such that if $\dconf(\ds,D)<\delta_2$, then
\[
 \Pr\Bb{\dconf(\ds',D_p)\ge \delta_1} \le e^{-\delta_2 n}.
\]
Hence, reducing $\delta$ if necessary, it follows that
if $\dconf(\ds,D)<\delta$ then
\begin{equation}\label{Lbig}
 \Pr\Bb{ |\vs| \ge (\rho(D)-\eps/4)n } \ge 1- e^{-\delta_1 n}-e^{-\delta_2 n} \ge 1-e^{-\delta n}.
\end{equation}
Note that in the argument above we could have sidestepped Corollary~\ref{cpij}, using a coloured version of
Theorem~\ref{lNgekD} and considering the coloured property `the red component of the
root has size at least $k$'. However, the approach above seems more intuitive
and we shall use Corollary~\ref{cpij} in Section~\ref{sec_extras}.

Let us call a vertex $v\in V(G)=V(G')$ \emph{usable} if it is incident with a blue
edge, i.e., an edge of $G''$. (These edges will be our `sprinkled' edges.)
Note that knowing $\ds'$ determines whether $v$ is usable: we don't know which edges
are present in $G''$, but we do know its degree sequence.
Our next aim is to find `enough' usable vertices in $\vs$, for which we need
some further definitions.

By the \emph{radius} $r(G)$ of a (locally finite) rooted graph $G$ we mean the
maximum distance of any vertex from the root, considering only vertices in the
component $C$ containing the root. Thus $r(G)$ is infinite if and only if $C$ is
infinite.

Given a coloured rooted graph $G$, we write $R(G)$ and $B(G)$ for the
red and blue subgraphs of $G$, respectively.
Let $\cG_t$ be the property of coloured rooted graphs $G$ that either

(i) $r(R(G))<t$ or

(ii) some vertex of $G$ within distance $t$ of the root is
incident with an edge of $B(G)$.

Note that, considering the shortest path to a
blue edge, (ii) is equivalent to (ii') some vertex of $R(G)$ within distance $t$
(in $R(G)$) of the root is incident with an edge of $B(G)$.
The property $\cG_t$ is clearly $(t+1)$-local.

Consider the case where $G=\cT_D\{p\}$ is a coloured rooted tree. Conditioning first
on the graph structure, if $r(G)<t$ then (i) will certainly hold. Otherwise,
there are at least $t$ edges of $G$ within distance $t$ of the root, and if any
one is blue (ii) holds. Thus
\[
 \Pr(\cT_D\{p\} \in \cG_t) \ge 1-p^t \ge 1-\eps/20.
\]

By Lemma~\ref{nobig} (with $\eps/2$ in place of $\eps$), there is some $\Delta$ such that
\[
 \Pr(\cT_{D_p}\in\cM_{\Delta,t})\ge 1-\eps/20.
\]
Let $\cH$ be the property
\[
 \cH = \{  R(G)\in \cM_{\Delta,t} \hbox{ and }G\in\cG_t \},
\] 
noting that
\begin{equation}\label{Ht}
 \Pr(\cT_D\{p\}\in \cH)\ge 1-\eps/10.
\end{equation}
We call a vertex $v$ of our coloured configuration model $G=\Gm_\ds\{p\}$
\emph{helpful} if $(G,v)\in \cH$, i.e., if $G$ rooted at
$v$ has property $\cH$. Let $H$ denote the set of helpful vertices.
From \eqref{Ht} and Theorem~\ref{th_localc}, if $\delta$ is chosen
small enough, then if $\dconf(\ds,D)<\delta$ we have
\begin{equation}\label{Usmall}
 \Pr\bb{ |H|\le n- \eps n/5 } \le e^{-\delta n}.
\end{equation}
Since, as noted above, knowing $\ds'$ determines which vertices are usable (incident
with edges of $G''$), it is easy to check from the definition of $\cH$ that
knowing $\ds$ (which is given), $\ds'$ and $G'$ determines which vertices of $G$ are helpful. 

From now on we condition on $\ds'$ and $G'$, assuming that
\begin{equation}\label{assump}
  |\vs|\ge (\rho(D)-\eps/4) n \hbox{\qquad and\qquad} |H|\ge n-\eps n/5.
\end{equation}
This makes sense since $\vs$ (the set of vertices 
in components of $G'$ with order at least $k$) and $H$ are determined by $\ds'$ and $G'$,
and \eqref{Lbig} and \eqref{Usmall} imply that the event \eqref{assump}
has probability at least $1-e^{-\delta n}$.

Suppose that $v\in \vs\cap H$. Then, since $v$ is helpful, every vertex in the
$t$-neighbourhood $\Ga_{\le t}^{G'}(v)$ of $v$ in $G'$ has degree at most
$\Delta$. Furthermore, from the definition of $\cG_t$ (recalling (ii') above),
either (a) the radius of $G'$ rooted at $v$ is at most $t-1$,
or (b) $\Ga_{\le t}^{G'}(v)$ meets an edge of $G''$, i.e., contains a usable vertex.
In case (a), it follows that the component of $v$ in $G'$ has at most
$1+\Delta+\cdots+\Delta^{t-1}<K$ vertices, contradicting $v\in \vs$.  Thus case
(b) holds and there is a path $P_v=v_0v_1\cdots v_r$ in $G'$ of length at most
$t$ where $v_0=v$, each $v_i$ has degree at most $\Delta$ in $G'$, and $v_r$ is
usable.

At this point we are finally ready to apply the sprinkling strategy
of Erd\H os and R\'enyi~\cite{ER_giant}.
Let us call a partition $(X,Y)$ of $\vs$ a \emph{potentially bad cut} if $|X|$,
$|Y|\ge \eps n/4$ and there are no edges of $G'$ joining $X$ to $Y$.
We call $(X,Y)$ a \emph{bad cut} if, in addition, no edge of $G''$ joins
$X$ to $Y$.  Since each
component of $G'$ in $\vs$ must lie either entirely in $X$ or entirely in $Y$, there are
at most
\begin{equation}\label{nopb}
 2^{|\vs|/k}\le 2^{n/k} \le e^{n/k} \le e^{\gamma n/2}
\end{equation}
potentially bad cuts, recalling that we chose $k\ge 2/\gamma$.

Let $(X,Y)$ be a potentially bad cut, and recall that $|H|\ge n-\eps n/5$. Thus
$X$ contains at least $\eps n/20$ helpful vertices $v$.  From each there is a
path $P_v$ as described above ending at some usable vertex $u$.  Because of the
degree conditions, at most $1+\Delta+\cdots+\Delta^t\le 2\Delta^t$ such paths
can end at a given usable vertex.  Since $P_v$ is a path in $G'$, and $X$ is a
union of components of $G'$, we conclude that $X$ contains at least $\alpha n$
usable vertices, where $\alpha=\eps/(40\Delta^t)$ as in \eqref{ag}.
Of course, the same applies to $Y$.

Recall that we have conditioned on $\ds'$ and $G'$, but not on $G''$.
In the configuration model corresponding to $G''$, each usable vertex
has at least one stub, so $X$ and $Y$ each correspond to sets of at least $\alpha n$
stubs. Since (if $\delta$ is chosen small enough) $G''$ has at most $\E(D)n$ edges,
by Lemma~\ref{join}
\[
 \Pr\bb{G''\hbox{ contains no }(X,Y)\hbox{ edge} \mid \ds',G'} \le e^{-\frac{\alpha^2n^2}{8\E(D)n}} = e^{-\gamma n}.
\]
From \eqref{nopb} it follows that the expected number of bad cuts (given $\ds'$ and $G'$) is at
most $e^{-\gamma n/2}$, so the probability that there are
any bad cuts is at most $e^{-\gamma n/2}$.  When there are no bad cuts, it is easy to
check that $L_1(G)\ge |\vs|-2\eps n/4 \ge (\rho(D)-3\eps/4)n$, completing the proof
of \eqref{MRaim} and hence of the multigraph case of Theorem~\ref{th2}.
\end{proof}

\section{Simple graphs}

As noted in the introduction,
Janson and Luczak~\cite{JL_new} proved a result that is similar
to the multigraph case of
Theorem~\ref{th2}: the assumptions are identical, but the error
bounds in the conclusions in~\cite{JL_new} are much weaker. An advantage of our stronger
error bounds is that they allow us to translate the result
to random \emph{simple} graphs without further restrictions on the degree sequences.
For this we need a simple lemma.

\begin{lemma}\label{psimple}
Let $D\in \cD$. Then for any $\eps>0$ there exists a $\delta>0$
such that if $\dconf(\ds,D)<\delta$ then
\[
 \Pr\bb{ \Gm_{\ds}\text{ is simple }} \ge e^{-\eps n},
\]
where $n=|\ds|$. 
Equivalently, if $D\in \cD$ and $\ds_n\to D$ in the sense that
\eqref{dcn1} and \eqref{dcn2} hold and $|\ds_n|\to\infty$, then
\[
 \Pr\bb{ \Gm_{\ds_n}\text{ is simple }} = e^{-o(|\ds_n|)}.
\]
\end{lemma}
In particular, the degree sequences we consider here are (for large $n$)
realizable by simple graphs.
\begin{proof}
The equivalence of the two statements follows easily from \eqref{coneq}; we
prove the first form.

Observe that there are constants $K$, $M$ and $\alpha>0$ such that, if $\delta$
is chosen small enough, then $\dconf(\ds,D)<\delta$ ensures that at least
$\alpha n$ vertices of $\ds$ have degree between 1 and $K$ (inclusive), and
$m=m(\ds)\le Mn$, where $n=|\ds|$ as usual. Indeed, choose any $K\ge 1$ such that $\Pr(D=K)>0$,
let $\delta\le \alpha=\Pr(D=K)/2$, and take $M=\E(D)/2+\alpha$, say.
These properties and \eqref{dctail}/\eqref{dconftail} are all that we need to know
about $\ds$.

Let $\cS$ denote the event that $\Gm_\ds$ is simple, and fix $\eps>0$.
Pick $\eta>0$ such that $\eta\log(4M/\alpha)\le \eps/2$ and $\eta\le \alpha/2$. By \eqref{dconftail}
there is a constant $C$, which we may take to be larger than $K$,
such that if $\delta$ is small enough, then at most $\eta n$ stubs are attached to
vertices of degree at least $C$. Let us call a vertex {\em low degree} if its degree
is between $1$ and $K$, and {\em high degree} if its degree is at least $C$.
Let $\cE$ be the event that the stubs attached to high degree vertices are paired
with stubs attached to {\em distinct} low degree vertices.

To determine whether $\cE$ holds, we test the at most $\eta n$ stubs attached to high
degree vertices one-by-one. At each stage, there are at least $\alpha n-\eta n\ge \alpha n/2$
low-degree vertices none of whose stubs has yet been paired. Since each such vertex has degree
at least one, and there are at most $2Mn$ unpaired stubs in total, it follows that
\[
 \Pr(\cE) \ge \left(\frac{\alpha n}{4Mn}\right)^{\eta n} \ge e^{-\eps n/2}.
\]
When $\cE$ holds, the graph $\Gm_\ds$ is simple if and only if the graph $G_0$
formed by the edges not incident with high-degree vertices is simple. But,
after revealing the partners of the stubs attached to the high-degree vertices,
the conditional distribution of $G_0$ is given by the configuration model
for some degree sequence in which all degrees are at most $C$, and at
least $\alpha n/2=\Theta(n)$ degrees are positive. The original
result of Bollob\'as~\cite{BB_config} (see also Bender and Canfield~\cite{BC})
thus gives $\Pr(\cS\mid \cE)=\Theta(1)$,
and the result follows.
\end{proof}

\begin{proof}[Proof of Theorem~\ref{th2} for $G_\ds$]
Let $\cP$ be any property of graphs.
Since the distribution of $\Gm_\ds$ conditioned
on the event $\cS$ that $\Gm_\ds$ is simple is exactly that of $G_\ds$, we have
\[
 \Pr( G_\ds \in \cP ) = \Pr( \Gm_\ds \in \cP \mid \Gm_\ds \in \cS) \le
 \frac{ \Pr(\Gm_\ds \in \cP)}{ \Pr(\Gm_\ds \in \cS)}.
\]
Fix $D\in \cD$. All statements about $\Gm_\ds$ in Theorem~\ref{th2} are of the form that
for some property $\cP$, there exist $\gamma,\delta_1>0$ such that if
$\dconf(\ds,D)<\delta_1$, then $\Pr(\Gm_\ds\in \cP)\le e^{-\gamma n}$.
(The theorem asserts this with $\delta_1=\gamma$.)
Lemma~\ref{psimple} gives us $\delta_2>0$ such that $\dconf(\ds,D)<\delta_2$
implies $\Pr(\Gm_\ds\in \cS)\ge e^{-\gamma n/2}$. Hence,
setting $\delta=\min\{\delta_1,\delta_2,\gamma/2\}$, if
$\dconf(\ds,D)<\delta$ then
\[
 \Pr( G_\ds \in \cP) \le e^{-\gamma n}/e^{-\gamma n/2} = e^{-\gamma n/2}\le e^{-\delta n},
\]
completing the proof of Theorem~\ref{th2}.
\end{proof}

As noted in the introduction, Theorem~\ref{th2} implies Theorem~\ref{th1}.

\section{Extensions}\label{sec_extras}

One of the motivations for studying the size of the largest component in the
configuration model $G_\ds$ is to consider percolation in this random
environment: given $0<p<1$, when does the random subgraph $G_\ds[p]$ of $G_\ds$
obtained by selecting edges independently with probability $p$ contain a giant
component? For example, Goerdt~\cite{Goerdt} showed that when $G_\ds$ is simply
a random $d$-regular graph, then there is a `threshold'
at $p=1/(d-1)$, above which a giant component appears.
As is by now well known, for results of the
present type this question turns out to be no more general than studying
$G_\ds$ directly (i.e., the case $p=1$), since one can view a random subgraph of
the configuration model as another instance of the configuration model. This is
discussed in detail by Fountoulakis~\cite{F_percd}; for a slightly different
approach see Janson~\cite{Janson_cperc}. We give the short argument since it is very 
easy with the ingredients we have to hand. In the next result we state
only the most interesting part formally; $D_p$ is the `$p$-thinned' version
of the probability distribution $D$, defined in \eqref{Dpdef} and 
appearing in Corollary~\ref{cpij}.

\begin{theorem}\label{th_perc}
Let $0<p<1$ be fixed. The conclusions of Theorems~\ref{th1} and~\ref{th2} hold
if $\Gm_\ds$ or $G_\ds$ is replaced by its random subgraph $\Gm_\ds[p]$ or
$G_\ds[p]$, and $\rho(D)$ and $\rho_k(D)$ are replaced by $\rho(D_p)$
and $\rho_k(D_p)$.

In particular, given $D\in \cD$ with $\Pr(D\ge 3)>0$, $0<p<1$ and $\eps>0$,
there exists $\delta>0$ such that, if $\dconf(\ds,D)<\delta$, then
\[
 \Pr\Bb{ \bigabs{L_1(G_\ds[p])-\rho(D_p) n} \ge \eps n} \le e^{-\delta n}
\]
and
\[
 \Pr\Bb{ \bigabs{L_1(\Gm_\ds[p])-\rho(D_p) n} \ge \eps n} \le e^{-\delta n},
\]
where $n=|\ds|$.
\end{theorem}
\begin{proof}
For $\Gm_\ds[p]$, this is essentially trivial from Theorem~\ref{th2} and
Corollary~\ref{cpij}. Indeed, by Theorem~\ref{th2} there exists $\delta_1>0$
such that if $\dconf(\ds_1,D_p)<\delta_1$ then
$\Gm_{\ds_1}$ has the desired property ($L_1$ close to $\rho(D_p)n$) with
probability at least $1-e^{-\delta_1 n}$. By Corollary~\ref{cpij} there is a $\delta$
such if $\dconf(\ds,D)<\delta$ then $\Pr(\dconf(\ds',D_p)<\delta_1)\ge 1-e^{-\delta n}$, where $\ds'$ is the degree sequence
of $\Gm_\ds[p]$. The result for $\Gm_\ds[p]$ follows by
noting that, conditional on $\ds'$, $\Gm_\ds[p]$ has the distribution of $\Gm_{\ds'}$.

For $G_\ds[p]$ we argue as in the last part of the previous section: note that conditional
on $\Gm_\ds$ being simple, $\Gm_\ds[p]$ has the same distribution as $G_\ds[p]$. Then
use Lemma~\ref{psimple} as before. The key point is that we do not try to condition
on $\Gm_\ds[p]$ being simple.
\end{proof}

\begin{remark}
Theorem~\ref{th_perc} implies that there is a `critical' $\pc$ such that $\Gm_\ds[p]$
has a giant component if and only if $p>\pc$. Indeed, $\pc=\inf\{p:\rho(D_p)=0\}$.
From basic branching process results, it is easy to see that $\pc=1/\E(Z_D)$, where
$Z_D$ is the distribution defined in \eqref{Zdef}. Either from this,
or from the fact that $\rho(D_p)>0$ if and only if $\E(D_p(D_p-2))>0$ it is
easy to see that
\[
 \pc = \frac{\E(D)}{\E(D(D-1))}.
\]
This is the same formula as given by Fountoulakis~\cite{F_percd}, for example, who
proved a form of Theorem~\ref{th_perc}, with stronger assumptions on the degree
sequences and weaker error bounds.
\end{remark}

\begin{remark}
Taking $|\ds_n|=n$ for notational simplicity, in the context of
Theorems~\ref{th1} and~\ref{th2}, the assumption that $\E(D)<\infty$,
corresponding to $m(\ds_n)=O(n)$, is very natural.  Indeed, it is not hard to
see that if $m(\ds_n)/n\to\infty$, then $\Gm_{\ds_n}$ will with high
probability contain a component with $n-o(n)$ vertices.  As soon as we
consider percolation on $\Gm_{\ds_n}$, however, it makes very good sense to allow $m(\ds_n)/n\to\infty$
and then study $\Gm_{\ds_n}[p_n]$ with $p_n\to 0$ as $n\to\infty$.
All we shall say here is that in many situations, for appropriate $p_n$,
the (random) degree sequence of $\Gm_{\ds_n}[p_n]$ will with high probability
be such that Theorem~\ref{th1} applies to it. For example, if all degrees are equal
to $k_n$ with $k_n\to\infty$ and $k_np_n\to\lambda\in \mathbb{R}$, then
the degree distribution of $\Gm_{\ds_n}[p_n]$ will be asymptotically Poisson
with mean $\lambda$. Hence Theorem~\ref{th1} can be used to show
that the threshold for percolation on $\Gm_{\ds_n}$ is at $\lambda=1$,
i.e., at $p_n=1/k_n$.
\end{remark}

Throughout the paper we have focussed on the number of vertices in the giant component.
What can we say about other properties of the giant component,
such as the number of vertices of given degree,
or the total number of edges? Results for these are given (under different
conditions) by Janson and Luczak~\cite{JL_new}, for example.  An often neglected
benefit of the branching-process viewpoint is that it typically gives
results of this type essentially automatically, not just for these properties,
but for any local property. (A version of this observation was made in a
different context by Bollob\'as, Janson and Riordan~\cite[Lemma 11.11]{kernels}; see also
\cite[Theorem 2.8]{k-core}.) 

We state the following result in a form analogous to Theorem~\ref{th2}; this of course
implies a version analogous to Theorem~\ref{th1}.

\begin{theorem}\label{th_local}
Let $\cP$ be a local property of rooted graphs, let $D\in \cD$ and let $\eps>0$.
There is some $\delta>0$ such that if $\dconf(\ds,D)<\delta$ then
the following hold, with $n=|\ds|$ and $G=\Gm_\ds$ or $G=G_\ds$:
\begin{equation}\label{totconc1}
 \Pr\Bb{ \bigabs{
 N_\cP(G)
  - n\Pr(\cT_D \hbox{ has }\cP )
} \ge \eps n }\le e^{-\delta n},
\end{equation}
and
\begin{equation}\label{giantconc}
 \Pr\Bb{ \bigabs{
 N_\cP(C_1)
  - n\Pr(\cT_D \hbox{ is infinite and has }\cP)
} \ge \eps n }\le e^{-\delta n},
\end{equation}
where $C_1$ is a component of $G$ of maximal order.
\end{theorem}

\begin{proof}
As usual, in the light of Lemma~\ref{psimple} we need only consider the case $G=\Gm_\ds$.
In this case, we have proved \eqref{totconc1} already in Theorem~\ref{th_local1}.

Turning to \eqref{giantconc}, let $D\in\cD$, $\eps>0$ and a
local property $\cP$ be given.  Let $\cS_k$ be the rooted-graph property `the
component of the root contains at least $k$ vertices', and $\cS_\infty$ `the
component of the root is infinite'. (We only consider the latter in the
context of $\cT_D$; all our graphs here are finite.)  Then, as $k\to\infty$, the
events $\{\cT_D\in \cS_k\} = \{|\cT_D|\ge k\}$ decrease to the event
$\{\cT_D\in\cS_\infty\} = \{\cT_D\hbox{ is infinite}\}$.  Hence
$\Pr(\cT_D\in\cS_k)\to\Pr(\cT_D\in\cS_\infty)$, and there is a constant $K$
such that for any $k\ge K$ we have
\begin{equation}\label{sksi}
 \Pr(\cT_D \in \cS_k\setminus \cS_\infty) < \eps/10.
\end{equation}
As before, let us say that an event holds `wvhp' if for some $\delta>0$
it holds with probability at least $1-e^{-\delta n}$ whenever
$\dconf(\ds,D)<\delta$.
By Lemma~\ref{lNgekD} there is some $k\ge K$ such that wvhp
\begin{equation}\label{numbig}
 \bigabs{N_{\ge k}(\Gm_\ds) - \rho(D) n} \le\eps n/10.
\end{equation}
Let $N=N_\cP(C_1)$ be the number of vertices we wish to count, i.e., those in
the largest component $C_1$ of $\Gm_\ds$ having property
$\cP$. Let $N'=N_{\cP\cap \cS_k}(\Gm_\ds)$ count vertices with property $\cP$
in components of size at least $k$. Then $N$ and $N'$ differ by at most $N_{\ge k}-L_1$,
which, by \eqref{numbig} and Theorem~\ref{th2}, is wvhp at most $\eps n/5$, say.
Applying \eqref{totconc1} to the local property $\cP\cap\cS_k$,
we deduce that wvhp $N$ is within $\eps n/4$ of $n\Pr(\cT_D\in \cP\cap \cS_k)$.
But by \eqref{sksi} this is within $\eps n/10$ of $n\Pr(\cT_D\in\cP\cap \cS_\infty)$,
establishing \eqref{giantconc}.
\end{proof}

For simple properties $\cP$,
it is easy to give explicit formulae for the probability that $\cT_D$ is infinite
and has property $\cP$. For example, if $\cP=\cP_d$ is the property that the root
has degree $d$, then defining $x_+$ as in Section~\ref{sec_surv}, the proof
of Theorem~\ref{rhoD} shows easily that
\[
 \Pr( \cT_D\hbox{ is infinite and has }\cP_d) = r_d(1-(1-x_+)^d).
\]
This gives an asymptotic formula for the number of degree-$d$ vertices in the giant
component $C_1$ that coincides with that of Janson and Luczak~\cite{JL_new}.

Rather than counting vertices with some local property, what happens if we want to
sum some `local function' $f(G,v)$ over vertices $v\in C_1$? Can we show that
\begin{equation}\label{fclose}
 n^{-1} \sum_{v\in C_1} f(C_1,v) \pto \E(f(\cT_D)) ?
\end{equation}
If $f$ is bounded then the answer is yes: simply express $f$ in terms of
indicator functions of local properties and apply Theorem~\ref{th_local}. In general,
\eqref{fclose} need not hold: for example, if $f(G,v)$ is the square of
the degree of $v$ then, since our assumptions give no control over $\sum_i d_i^2$,
\eqref{fclose} can fail.

Suppose that $f(G,v)$ is the degree of $v$, so $\sum_{v\in C_1} f(C_1,v)$ is
twice the number of edges in the giant component. Then, by \eqref{dconftail},
for any $\eps>0$ there is a $C$ such that if $\dconf(\ds,D)$ is small enough,
then
\[
  \sum_{v\in C_1 : d_{C_1}(v)\ge C} f(C_1,v)
 \le  \sum_{v\in G: d_G(v)\ge C} f(G,v) \le \eps n,
\]
and considering the bounded function obtained by truncating $f$ at $C$,
we see that \eqref{fclose} holds in this case, even though $f$ is unbounded.
A similar argument can be applied to other unbounded $f$, leading to results
concerning, for example, the number edges in the
giant component between vertices of degree~2 and degree~3. We omit the details.

\end{document}